\documentclass[preprint]{elsarticle}

\usepackage{amssymb}
\usepackage{amsmath}
\usepackage{amsthm}
\usepackage{enumerate}
\usepackage{graphicx}
\usepackage{amsfonts}

\journal{...}

\newtheorem{theorem}{Theorem}[section]

\newtheorem{lemma}{Lemma}[section]

\newtheorem{cor}{Corollary}[section]
\newtheorem{prop}{Proposition}[section]
\newtheorem{remark}{Remark}[section]

\numberwithin{equation}{section}

\theoremstyle{definition}

\begin{document}

	\begin{frontmatter}

		\title{Perturbed Bernstein-type operators}
		
		\author[1]{Ana-Maria Acu}
		\author[2]{Heiner Gonska}
		\address[1]{Lucian Blaga University of Sibiu, Department of Mathematics and Informatics, Str. Dr. I. Ratiu, No.5-7, RO-550012  Sibiu, Romania, e-mail: anamaria.acu@ulbsibiu.ro}
		\address[2]{University of Duisburg-Essen, Faculty of Mathematics, Bismarckstr. 90, 47057 Duisburg, Germany, e-mail: heiner.gonska@uni-due.de }

		\begin{abstract}
		The present paper deals with  modifications of  Bernstein, Kantorovich, Durrmeyer  and genuine Bernstein-Durrmeyer operators. Some previous results are improved in this study. Direct estimates for these operators by means of the first and second  modulus of continuity are given. Also the asymptotic formulas for the new operators are proved.
			\end{abstract}

		
	\end{frontmatter}
\section{Introduction }
In  2018  Khosravian-Arab, Dehghan and Eslahchi introduced three modifications of the classical Bernstein operator. In this note we follow their approach, explain it and discuss further relevant, but truly different Bernstein-type operators which have been attracting attention in the past. Thus we will discuss the modifications of the classical Bernstein operators (pointwise defined, preserve linear functions, but not commutative), classical Kantorovich operators (defined on $L_1$, do not preserve linear functions), Durrmeyer operators (globally defined, commutative, do not preserve  linear functions) and genuine Bernstein-Durrmeyer operators (globally defined, also commutative, preserve linear functions).
Only in the Bernstein case we will go one step further and add remarks on a second perturbation created by modifying the classical recursion twice.

The organization of this note follows the lines given above. Before we will give estimates we add two short sections on the recursion for the fundamental functions of the Bernstein operator and on the use of $\omega_2$.

\section{On the recursion for the fundamental functions of Bernstein operators}
For $f\in C[0,1]$ the Bernstein operator $B_n:C[0,1]\to\prod_n$ is given by
$$ B_n(f;x)=\displaystyle\sum_{k=0}^np_{n,k}(x)f\left(\dfrac{k}{n}\right),\,\,x\in[0,1], $$
where the fundamental functions are defined by
$$ p_{n,k}(x):=\left\{\begin{array}{l}\displaystyle{n\choose k}x^k(1-x)^{n-k},\,\,0\leq k\leq n,\,\, x\in[0,1],\\
\vspace{-0.2cm}\\
0,\,\,k<0\textrm{ or }n<k.\end{array}\right. $$
It is well-known that these functions satisfy the recursion
\begin{equation}\label{*}
p_{n,k}(x)=(1-x)p_{n-1,k}(x)+xp_{n-1,k-1}(x),\,\,0\leq k\leq n.
\end{equation}
In particular,
\begin{align*}
&p_{n,0}(x)=(1-x)p_{n-1,0}(x)=(1-x)^n,\\
&p_{n,n}(x)=xp_{n-1,n-1}(x)=x^n.
\end{align*}
This recursion is closely related to the so-called de Casteljau algorithm and other  methods to compute  a value $B_n(f;x)$, $x$ fixed. See \cite{GoLu} for details.

In  \cite{1} the recursion form (\ref{*}) is perturbed by replacing it in the first modification $B_n^{M,1}$ by
\begin{align*}
&p_{n,k}^{M,1}(x)=a(x,n)p_{n-1,k}(x)+a(1-x,n)p_{n-1,k-1}(x),1\leq k\leq n-1,\\
&p_{n,0}^{M,1}(x)=a(x,n)(1-x)^{n-1},\,\,p_{n,n}^{M,1}(x)=a(1-x,n)x^{n-1}.
\end{align*}
Here
$$a(x,n)=a_1(n)x+a_0(n),\,\, n=0,1,\dots, $$
replaces $(1-x)$ in the original formula. In the  papers dealing with this modification (see \cite{AcAg}, \cite{ana3},  \cite{Opris}) the superscript "$M,1$" refers to this first disorder in the recursion.

If we carry out the original recursion once again, we obtain
\begin{align*}
p_{n,k}(x)&=(1-x)p_{n-1,k}(x)+xp_{n-1,k-1}(x)\\
&=(1-x)^2p_{n-2,k}(x)+2x(1-x)p_{n-2,k-1}(x)+x^2p_{n-2,k-2}(x).
\end{align*}
So for the second modification $B_n^{M,2}$
\begin{align*}
(1-x)^2 &\textrm{ is replaced by } b(x,n)=b_2(n)x^2+b_1(n)x+b_0(n),\\
2x(1-x) &\textrm{ is replaced by } d_0(n)x(1-x),\textrm{ and}\\
x^2 &\textrm{ is replaced by } b(1-x,n).
\end{align*}
Fortunately enough, this disorder is always introduced in the last step/two steps only. This means that the first $n-1/n-2$ fundamental functions remain intact, and a somewhat arbitrary perturbation is only introduced in the last/ last two step(s).

The present note is mostly written with the ambition to show how resistant the fundamental functions $p_{n,k}$ are with respect to such unexpected intrusions.

\section{On the use of $\omega_2$}

For many years researchers in approximation theory have been striving to give inequalities with  $\omega_2$  being the dominant expression. Many people simply still ignore this. In this section we will provide a very brief explanation why the use of $\omega_2$ (or related quantities such as the Ditzian-Totik modulus of second order) is indeed the better and more powerful tool from the quantitative point of view.

Exemplarily we will discuss the classical Bernstein operator $B_n$ and start with a very good results by P\u alt\u anea \cite{Pa} who confirmed an earlier conjecture of the second author \cite{G11}, namely that one has
$$ \|B_nf-f\|_{\infty}\leq 1\cdot\omega_2\left(f;\dfrac{1}{\sqrt{n}}\right),\,\,f\in C[0,1],\,\,n\in\mathbb{N}. $$
Here the constant $1$ is best possible.

This implies that the approximation by $B_n$ is of order ${\cal O}\left(\dfrac{1}{n}\right)$ for $f\in C^2[0,1]$, and of order ${\cal O}\left(\dfrac{1}{\sqrt{n}}\right)$ for $f\in C^1[0,1]$, and even for $f\in Lip1=\left\{f\in C[0,1]:\omega_1(f;t)={\cal O}(t)\right\}$. An estimate in terms of $c\cdot\omega_1\left(f;\dfrac{1}{\sqrt{n}}\right)$ only reaches $Lip1$. If $\omega_1(f;t)=o(t)$, then $f$ is a constant.

However, the inequality in terms of $\omega_2$ also shows that one has ${\cal O}\left(\dfrac{1}{\sqrt{n}}\right)$ for $$f\in Lip^{*}1=\left\{f\in C[0,1]:\omega_2(f;t)={\cal O}(t)\right\}.$$
Moreover, $Lip^{*}1 \supsetneqq Lip 1$, so the same order is true for the bigger set $Lip^{*} 1$. An example of a function $g\in Lip^{*}1\setminus Lip 1$ is
$$ g(x)=\left\{\begin{array}{l}0,\,\,x=0,\\
\vspace{-0.2cm}\\
x\log|x|,\,\,0<x\leq 1.\end{array}\right. $$
The problem is at $x=0$. Moreover,
$$ Lip^{*}1\subset\left\{f\in C[0,1]: \omega_1(f;\delta)={\cal O}(\delta\cdot|\log \delta|),\,\delta\to 0 \right\}, $$
the Dini-Lipschitz class. Hence it follows that
$$ Lip^{*} 1 \subset Lip\alpha,\,\, 0<\alpha<1. $$
All this happens inside $C[0,1]$. For $k\geq 1$ this story repeats between $C^{k}[0,1]$ and $C^{k+2}\subset C^{k+1}$, a fact being important when dealing with simultaneous approximation. Much more can be found in the seminal paper of Zygmund \cite{Zyg}.

	\section{The modified Bernstein operator $B_n^{M,1}$}

Recently, H. Khosravian-Arab et al. \cite{1} have introduced  modified Bernstein operators  as follows:
\begin{eqnarray}\label{ne1}
B_n^{M,1}(f,x) = \sum\limits_{k=0}^{n} p_{n,k}^{M,1}(x) \ f\left(\frac{k}{n} \right) , \ \ x \in [0,1].
\end{eqnarray}

Note that throughout the paper we will assume that $B_n^{M,1}(e_0,x)=1$, namely the
sequences $a_i(n), i=0,1$, verify the condition
\begin{equation}\label{A}
2a_0(n)+a_1(n)=1.
\end{equation}


\begin{theorem}
	For $B_n^{M,1}$ given above, $f\in C[0,1]$, $x\in[0,1]$, $n\geq 1$, we have
	$$ |B_n^{M,1}(f;x)-f(x)|\leq |B_n(f;x)-f(x)|+\left|(1+a_1(n))\left(\frac{1}{2}-x\right)\right|\omega_1\left(f;\frac{1}{n}\right).  $$
\end{theorem}
\begin{proof} We have
	\begin{equation}\label{A1}|B_n^{M,1}(f;x)-f(x)|\leq |B_n(f;x)-f(x)|+|B_n^{M,1}(f;x)-B_n(f;x)|. \end{equation}
	In the following we will give an estimate of the quantity $|B_n^{M,1}(f;x)-B_n(f;x)|$. So,
	\begin{align*}
B_n^{M,1}(f;x)-B_n(f;x)&=\displaystyle\sum_{k=0}^n\{a(x,n)p_{n-1,k}(x)+a(1-x,n)p_{n-1,k-1}(x)\}f\left(\dfrac{k}{n}\right)\\
	&-\displaystyle\sum_{k=0}^n\{(1-x)p_{n-1,k}(x)+xp_{n-1,k-1}(x)\}f\left(\dfrac{k}{n}\right)\\
	&=\displaystyle\sum_{k=0}^{n-1}\{(a_1(n)\!+\!1)x\!+\!a_0(n)\!-\!1\}p_{n-1,k}(x)f\left(\dfrac{k}{n}\right)\\
&	+\displaystyle\sum_{k=1}^n\{(-1-a_1(n))x+a_1(n)+a_0(n)\}p_{n-1,k-1}(x)f\left(\dfrac{k}{n}\right)\\
		&=\displaystyle\sum_{k=0}^{n-1}\{(a_1(n)+1)x+a_0(n)-1\}p_{n-1,k}(x)f\left(\dfrac{k}{n}\right)\\
&	+\displaystyle\sum_{k=0}^{n-1}\{(-(1+a_1(n))x+a_1(n)+a_0(n)\}p_{n-1,k}(x)f\left(\dfrac{k+1}{n}\right)\\
	&=\displaystyle\sum_{k=0}^{n-1}\left[f\left(\dfrac{k+1}{n}\right)-f\left(\dfrac{k}{n}\right)\right]\left\{-(1+a_1(n))x+a_0(n)+a_1(n)\right\}p_{n-1,k}(x).
		\end{align*}
				Therefore,
		\begin{align*} \left|B_n^{M,1}(f;x)-B_n(f;x)\right|&\leq |-(1+a_1(n))x+a_0(n)+a_1(n)|\omega_1\left(f;\dfrac{1}{n}\right)\\
		&=\left|(1+a_1(n))\left(\dfrac{1}{2}-x\right)\right|\omega_1\left(f;\dfrac{1}{n}\right) \end{align*}
		and replacing this estimate in (\ref{A1}) the proof is complete.
	\end{proof}
\begin{remark} \begin{itemize}
		\item[i)] For $a_1(n)=-1$ all the estimates for Bernstein operator $B_n$ hold.
		\item[ii)] If $a_1(n)$ is bounded, say $|a_1(n)|\leq A_1$, then $$|B_n^{M,1}(f;x)-f(x)|\leq |B_n(f;x)-f(x)|+\dfrac{1}{2}(1+A_1)\omega_1\left(f;\dfrac{1}{n}\right).$$
		\item[iii)] If $f\in C^2[0,1]$, then for $a_1(n)$ bounded $\|B_n^{M,1}(f)-f\|_{\infty}={\cal O}\left(\dfrac{1}{n}\right)$. This result is an improvement of \cite[Theorem 9]{1}.
	\end{itemize}
	\end{remark}
In order to prove a quantitative Voronovskaja theorem for $B_n^{M,1}$ we first identify the limit.
\begin{prop} Suppose that $B_n^{M,1}$ is given as above, $f\in C^2[0,1]$, $x\in[0,1]$ and \linebreak $L_1:=\displaystyle\lim_{n\to\infty} a_1(n)$ exists. Then
	$$\displaystyle\lim_{n\to \infty} n\left[B_n^{M,1}(f;x)-f(x)\right]=\displaystyle\frac{x(1-x)}{2}f^{\prime\prime}(x)+\dfrac{1-2x}{2}(1+L_1)f^{\prime}(x).  $$
	\end{prop}
\begin{proof}
	As above write
	\begin{align*}
	n\left[B_n^{M,1}(f;x)-f(x)\right]&=n\left[B_n(f;x)-f(x)\right]+n\left[B_n^{M,1}(f;x)-B_n(f;x)\right]\\
	&=:T_1(x)+T_2(x).
	\end{align*}
	The limit of $T_1(x)$ is known, i.e., $\dfrac{x(1-x)}{2}f^{\prime\prime}(x)$. Moreover,
	\begin{align*}
	T_2(x)&=\left[-(1+a_1(n))x+a_0(n)+a_1(n)\right]\displaystyle\sum_{k=0}^{n-1}n\left[f\left(\dfrac{k+1}{n}\right)-f\left(\dfrac{k}{n}\right)\right]p_{n-1,k}(x)\\
	&=\left[-(1+a_1(n))x+a_0(n)+a_1(n)\right]\left(B_nf\right)^{\prime}(x).
	\end{align*}
	Hence,
	$$ n\left[B_n^{M,1}(f;x)-f(x)\right]=n\left[B_n(f;x)-f(x)\right] +\left(1+a_1(n)\right)\left(-x+\dfrac{1}{2}\right)\left(B_nf\right)^{\prime}(x). $$
	Since $\displaystyle\lim_{n\to\infty}n\left[B_n(f;x)-f(x)\right]=\dfrac{x(1-x)}{2}f^{\prime\prime}(x)$ and $\displaystyle\lim_{n\to\infty}\left(B_nf\right)^{\prime}(x)=f^{\prime}(x)$, the proof is complete.
\end{proof}
\begin{theorem}
	Suppose that $B_n^{M,1}$ is given as above, $f\in C^2[0,1]$, $L_1=\displaystyle\lim_{n\to\infty}a_1(n)$ exists. Then for $x\in[0,1]$ there holds
	\begin{align*}
	\Delta_n^B&:=\left|n\left[B_n^{M,1}(f;x)-f(x)\right]-\displaystyle\frac{x(1-x)}{2}f^{\prime\prime}(x)-\dfrac{1-2x}{2}(1+L_1)f^{\prime}(x)\right|\\
	&\leq X\left\{\dfrac{5}{6}\dfrac{|X^{\prime}|}{\sqrt{3(n-2)X+1}}\omega_1\left(f^{\prime\prime};\sqrt{\dfrac{3(n-2)X+1}{n^2}}\right)+\dfrac{13}{16}\omega_2\left(f^{\prime\prime};\sqrt{\dfrac{3(n-2)X+1}{n^2}}\right)\right\}\\
	&+\dfrac{|X^{\prime}|}{2}\left\{|L_1-a_1(n)|\cdot \| f^{\prime} \|_{\infty}+ |1+L_1|\cdot\left(\dfrac{13}{4}\omega_2\left(f^{\prime};\dfrac{1}{\sqrt{n}}\right)+\dfrac{1}{\sqrt{n}}\omega_1\left(f^{\prime};\dfrac{1}{\sqrt{n}}\right)\right) \right\}.
		\end{align*}
		Here $X:=x(1-x)$, i.e., $X^{\prime}=1-2x$.
\end{theorem}
\begin{proof}
	For $\Delta_n^B$ the following inequality holds
	\begin{align}\label{ec1} \Delta_n^B&\leq \left|n\left[B_n(f;x)-f(x)\right]-\dfrac{x(1-x)}{2}f^{\prime\prime}(x) \right|\nonumber\\
	&+
	\left|n\left[B_n^{M,1}(f;x)-B_n(f;x)\right]-\dfrac{1-2x}{2}(1+L_1)f^{\prime}(x)\right|. \end{align}
	Gonska and Ra\c sa \cite{GonskaRasa} obtained a Voronovskaya estimate with first and second modulus of smoothness for Bernstein operator  as follows
	
	\begin{align}
	&\left|n\left[B_n(f;x)-f(x)\right]-\displaystyle\frac{x(1-x)}{2}f^{\prime\prime}(x)\right|\nonumber\\
	&\leq X\left\{\dfrac{5}{6}\dfrac{|X^{\prime}|}{\sqrt{3(n-2)X+1}}\omega_1\left(f^{\prime\prime};\sqrt{\dfrac{3(n-2)X+1}{n^2}}\right)+\dfrac{13}{16}\omega_2\left(f^{\prime\prime};\sqrt{\dfrac{3(n-2)X+1}{n^2}}\right)\right\}.\label{ec2}
	\end{align}
	We estimate the second difference of (\ref{ec1}) as follows
	\begin{align}	&\left|n\left[B_n^{M,1}(f;x)-B_n(f;x)\right]-\dfrac{1-2x}{2}(1+L_1)f^{\prime}(x)\right|\nonumber\\
	&=\left|(1+a_1(n))\left(-x+\dfrac{1}{2}\right)\left(B_nf\right)^{\prime}(x)-\dfrac{1-2x}{2}(1+L_1)f^{\prime}(x)\right|\nonumber\\
	&=\left|\dfrac{1-2x}{2}\left(a_1(n)-L_1\right)\left(B_nf\right)^{\prime}(x)-\dfrac{1-2x}{2}(1+L_1)\left[f^{\prime}(x)-\left(B_nf\right)^{\prime}(x)\right]\right|\nonumber\\
		&\leq\dfrac{|1-2x|}{2}\left\{|L_1-a_1(n)|\cdot |\left(B_nf\right)^{\prime}(x)|+|1+L_1|\left|f^{\prime}(x)-\left(B_nf\right)^{\prime}(x)\right|\right\}.\label{ec3}
	\end{align}
	Moreover, we use (see \cite[Theorem 4.1]{Gonska})
	\begin{equation}\label{ec4}
	\left|f^{\prime}(x)-\left(B_nf\right)^{\prime}(x)\right|\leq \dfrac{13}{4}\omega_2\left(f^{\prime};\dfrac{1}{\sqrt{n}}\right)+\dfrac{1}{\sqrt{n}}\omega_1\left(f^{\prime};\dfrac{1}{\sqrt{n}}\right).
	\end{equation}
	Also, we have
	\begin{equation}
	\label{ec5}\left|(B_nf)^{\prime}(x)\right|=\left|\displaystyle\sum_{k=0}^{n-1}\dfrac{f\left(\dfrac{k+1}{n}\right)-f\left(\dfrac{k}{n}\right)}{\frac{1}{n}}p_{n-1,k}(x)\right|\leq n\omega_1\left(f;\dfrac{1}{n}\right)\leq\| f^{\prime}\|_{\infty}.
	\end{equation}
	Using the relations (\ref{ec1})-(\ref{ec5}) the theorem is proved.
\end{proof}

\begin{cor} We have
	$$\Delta_n^B\leq\left\{\begin{array}{l} {\cal O}\left(\dfrac{1}{\sqrt{n}}\right)+\dfrac{1}{2}\left|L_1-a_1(n) \right|\cdot \| f^{\prime}\|_{\infty},\textrm{ for } f\in C^3[0,1],\\
	\vspace{-0.4cm}\\
	{\cal O}\left(\dfrac{1}{n}\right) +\dfrac{1}{2}\left|L_1-a_1(n) \right|\cdot \| f^{\prime}\|_{\infty},\textrm{ for } f\in C^4[0,1].
		\end{array}\right.  $$
	\end{cor}
Theorem 9 in \cite{1} should be reformulated in the following way.
\begin{prop}
	If $B_n^{M,1}$ is given as above (positive or non-positive), then for $f\in B[0,1]$ (bounded functions) holds
	$$ \| B_n^{M,1}(f)-f\|_{\infty}\leq 2\left(3|a_1(n)|+1\right)\omega_1\left(f;\dfrac{1}{\sqrt{n}}\right),\,n\geq 3. $$
	If $(a_1(n))$ is bounded, then 
	$$ \| B_n^{M,1}(f)-f\|_{\infty}={\cal O}(1)\omega_1\left(f;\dfrac{1}{\sqrt{n}}\right),  $$
	and 
	$$ \| B_n^{M,1}(f)-f\|_{\infty}={ o}(1),\textrm{ if } f\in C[0,1]. $$
\end{prop}

\section{The modification $B_n^{M,2}$}

 Khosravian-Arab et al. \cite{1} also introduced a second modification of the Bernstein operator as follows:
\begin{align}\label{TK2}
{B}_n^{M,2}(f;x)&=\displaystyle\sum_{k=0}^{n} {p}_{n,k}^{M,2}(x)f\left(\dfrac{k}{n}\right),
\end{align}
where
\begin{align*}
{p}_{n,k}^{M,2}(x)&=\left[\dfrac{n}{2}x^2+\left(-1-\dfrac{n}{2}\right)x+1\right]p_{n-2,k}(x)+nx(1-x)p_{n-2,k-1}(x)\\
&+\left[\dfrac{n}{2}x^2+\left(-\dfrac{n}{2}+1\right)x\right]p_{n-2,k-2}(x). \end{align*}

\begin{lemma}
	The moments of the  operators $B_n^{M,2}$ are given by
	\begin{itemize}
		\item[i)] ${B}_n^{M,2}(e_0;x)=1$;
		\item[ii)] ${B}_n^{M,2}(e_1;x)=x$;
		\item[iii)] ${B}_n^{M,2}(e_2;x)=x^2+\displaystyle\frac{2x(1-x)}{n^2}$.
	\end{itemize}
\end{lemma}
\begin{theorem}
	Let $B_n^{M,2}$ be the modified Bernstein operator defined in (\ref{TK2}).  Then for $f\in C[0,1]$ there holds
	$$ \left\|B_n^{M,2}-f\right\|_{\infty}\leq\left\{ \begin{array}{l}
	\dfrac{1}{8}\omega_1\left(f^{\prime};\dfrac{1}{n}\right)+\dfrac{1}{\sqrt{n-2}}\omega_1\left(f^{\prime};\dfrac{1}{\sqrt{n-2}}\right)+\dfrac{2}{n}\| f^{\prime}\|_{\infty}=o(1),\,\, f\in C^1[0,1],\\
	\vspace{-0.4cm}\\
	{\cal O}\left(\dfrac{1}{n}\right),\,\,f\in C^2[0,1].
		\end{array} \right.  $$
\end{theorem}

\begin{proof}
	We have 
	\begin{align*}
	\left|B_n^{M,2}(f;x)-f(x)\right|
&	=\left|\displaystyle\sum_{k=0}^{n-2}p_{n-2,k}(x)\left\{\dfrac{nx(x-1)}{2}\left[f\left(\dfrac{k}{n}\right)-2f\left(\dfrac{k+1}{n}\right)+f\left(\dfrac{k+2}{n}\right)\right]\right.\right.\\
	&\left.\left.+(1-x)f\left(\dfrac{k}{n}\right)+xf\left(\dfrac{k+2}{n}\right)-f(x)\right\}\right|\\
	&	=\left|\displaystyle\sum_{k=0}^{n-2}p_{n-2,k}(x)\left\{\dfrac{nx(x-1)}{2}\left[f\left(\dfrac{k}{n}\right)-2f\left(\dfrac{k+1}{n}\right)+f\left(\dfrac{k+2}{n}\right)\right]\right.\right.\\
	&+(1-x)\left[f\left(\dfrac{k}{n-2}\right)-f(x)\right]+(1-x)\left[f\left(\dfrac{k}{n}\right)-f\left(\dfrac{k}{n-2}\right)\right]\\
	&\left.\left.+x\left[f\left(\dfrac{k}{n-2}\right)-f(x)\right]
	+x\left[f\left(\dfrac{k+2}{n}\right)-f\left(\dfrac{k}{n-2}\right)\right]\right\}\right|.
	\end{align*}
	Using the relation (see \cite{Pa})
	 $$ \| B_n(f)-f\|_{\infty}\leq \omega_2\left(f;\dfrac{1}{\sqrt{n}}\right),$$
	 we obtain
	 \begin{align*}
	 \left|B_n^{M,2}(f;x)-f(x)\right|&\leq \dfrac{nx(1-x)}{2}\omega_2\left(f;\dfrac{1}{n}\right)+(1-x)\omega_2\left(f;\dfrac{1}{\sqrt{n-2}}\right)+x\omega_2\left(f;\dfrac{1}{\sqrt{n-2}}\right)\\
	 &+(1-x)\omega_1\left(f;\dfrac{2}{n}\right)+x\omega_1\left(f;\dfrac{2}{n}\right)\\
	 &=\dfrac{nx(1-x)}{2}\omega_2\left(f;\dfrac{1}{n}\right)+\omega_2\left(f;\dfrac{1}{\sqrt{n-2}}\right)+\omega_1\left(f;\dfrac{2}{n}\right)\\
	 &\leq \dfrac{n}{8}\omega_2\left(f;\dfrac{1}{n}\right)+\omega_2\left(f;\dfrac{1}{\sqrt{n-2}}\right)+\omega_1\left(f;\dfrac{2}{n}\right),\textrm{ for } f\in C[0,1],
	 \end{align*}
	 and the theorem is proved.
\end{proof}
\begin{remark}
	The above inequality is an improvement and a generalization of \cite[Theorem 14]{1}. There a non-quantitative statement is obtained for $f\in C^2[0,1]$ only.
\end{remark}

\section{The modified  Kantorovich operators $K_n^{M,1}$}

An integral modification of Bernstein operators  was introduced by Kantorovich \cite{K} as follows:

\begin{equation}
\label{K}
K_{n}(f;x)=(n+1)\displaystyle\sum_{k=0}^n{p}_{n,k}(x)\int_{\frac{k}{n+1}}^{\frac{k+1}{n+1}}f(t)dt.
\end{equation}
In a recent article of the present authors \cite{AcGo} the results on these mapping were supplemented. Here we only mention one result from there.
Applying P\u alt\u anea's result \cite[Corollary 2.2.1]{Pa} the following estimate in terms of the first and second modulus of continuity for the classical Kantorovich operators is obtained.

\begin{theorem}
	For $n\geq 1$ and all $f\in C[0,1]$ there holds
	$$ \| K_nf-f\|_{\infty}\leq \dfrac{1}{2\sqrt{n+1}}\omega_1\left(f;\dfrac{1}{\sqrt{n+1}}\right)+\dfrac{9}{8}\omega_2\left(f;\dfrac{1}{\sqrt{n+1}}\right).
	$$
\end{theorem}

Recently, a Kantorovich variant of  the modified
Bernstein operators (\ref{ne1}) was investigated in \cite{ana3}. These operators are given by
$$  K_n^{M,1}(f;x):=(n+1)\sum_{k=0}^np_{n,k}^{M,1}(x)\int_{\frac{k}{n+1}}^{\frac{k+1}{n+1}}f(t)dt. $$
A certain Stancu modification was  introduced by Opri\c s \cite{Opris}.

\begin{theorem}
	For $K_n^{M,1}$ given above, $f\in C[0,1]$, $x\in[0,1]$, $n\geq 1$, we have
	$$ |K_n^{M,1}(f;x)-f(x)|\leq |K_n(f;x)-f(x)|+\left|(1+a_1(n))\left(\frac{1}{2}-x\right)\right|\omega_1\left(f;\frac{1}{n+1}\right).  $$
\end{theorem}

\begin{proof}
Again we start with 
	\begin{equation}\label{y1} \left| K_n^{M,1}(f;x)-f(x) \right|\leq \left| K_n(f;x)-f(x) \right|+\left| K_n^{M,1}(f;x)-K_n(f;x) \right|. \end{equation}
	In the following we will estimate the quantity $\left| K_n^{M,1}(f;x)-K_n(f;x) \right|$.  
	
One has
	\begin{align*}
	& K_n^{M,1}(f;x)-K_n(f;x)=(n+1)\sum_{k=0}^n\left\{\left(a(x,n)p_{n-1,k}(x)+a(1-x,n)p_{n-1,k-1}(x)\right)\int_{\frac{k}{n+1}}^{\frac{k+1}{n+1}}f(t)dt\right.\\
	&\left.-\left((1-x)p_{n-1,k}(x)+xp_{n-1,k-1}(x)\right)\int_{\frac{k}{n+1}}^{\frac{k+1}{n+1}}f(t)dt\right\}\\
	&=(n+1)\sum_{k=0}^{n-1}\left[(a_1(n)+1)x-(a_0(n)+a_1(n))\right]p_{n-1,k}(x)\int_{\frac{k}{n+1}}^{\frac{k+1}{n+1}}f(t)dt\\
	&-(n+1)\sum_{k=0}^{n-1}\left[(a_1(n)+1)x-(a_0(n)+a_1(n))\right]p_{n-1,k}(x)\int_{\frac{k+1}{n+1}}^{\frac{k+2}{n+1}}f(t)dt\\
	&=(n+1)\sum_{k=0}^{n-1}\left[-(a_1(n)+1)x+(a_0(n)+a_1(n))\right]p_{n-1,k}(x)\left[\int_{\frac{k+1}{n+1}}^{\frac{k+2}{n+1}}f(t)dt-\int_{\frac{k}{n+1}}^{\frac{k+1}{n+1}}f(t)dt\right]\\
	&=(n+1)\sum_{k=0}^{n-1}\left[-(a_1(n)+1)x+(a_0(n)+a_1(n))\right]p_{n-1,k}(x)\int_{\frac{k+1}{n+1}}^{\frac{k+2}{n+1}}\left[f(t)-f\left(t-\frac{1}{n+1}\right)\right]dt.
	\end{align*}
	Therefore,
	\begin{equation}\label{y2} \left| K_n^{M,1}(f;x)-K_n(f;x) \right|\leq \left|(1+a_1(n))\left(\frac{1}{2}-x\right)\right|\omega_1\left(f;\frac{1}{n+1}\right). \end{equation}
	From (\ref{y1}) and (\ref{y2}) it follows that for all cases of $K_n^{M,1}$ (positive and non-positive) we have
	$$ |K_n^{M,1}(f;x)-f(x)|\leq |K_n(f;x)-f(x)|+\left|(1+a_1(n))\left(\frac{1}{2}-x\right)\right|\omega_1\left(f;\frac{1}{n+1}\right).  $$
\end{proof}

\begin{remark} \begin{itemize}
		\item[i)] For $a_1(n)=-1$ all the estimates for Kantorovich operator $K_n$ hold.
		\item[ii)] If $a_1(n)$ is bounded, say $|a_1(n)|\leq A_1$, then $$|K_n^{M,1}(f;x)-f(x)|\leq |K_n(f;x)-f(x)|+\dfrac{1}{2}(1+A_1)\omega_1\left(f;\dfrac{1}{n+1}\right).$$
		\item[iii)] If $f\in C^2[0,1]$, then for $a_1(n)$ bounded $\|K_n^{M,1}(f)-f\|_{\infty}={\cal O}\left(\dfrac{1}{n}\right)$. This result is an improvement of \cite[Theorem 2.6]{ana3}.
	\end{itemize}
\end{remark}

We will give next a Voronovskaya-type result for the modifications $K_n^{M,1}$.
\begin{theorem}
	Suppose that  $f\in C^2[0,1]$ and $L_1=\displaystyle\lim_{n\to\infty}a_1(n)$ exists. Then for $x\in[0,1]$ there holds
	\begin{align*}
	&\Delta_n^K:=\left|n\left[K_n^{M,1}(f;x)-f(x)\right]-\displaystyle\frac{X}{2}f^{\prime\prime}(x)-\dfrac{X^{\prime}}{2}(2+L_1)f^{\prime}(x)\right|\\
	&\leq \dfrac{2}{3(n+1)}\left(\dfrac{3}{4}\|f^{\prime}\|_{\infty}+\|f^{\prime\prime}\|_{\infty}\right)\nonumber
	+\dfrac{9}{32}\left\{\dfrac{2}{\sqrt{n+1}}\omega_1\left(f^{\prime\prime};\dfrac{1}{\sqrt{n+1}}\right)+\omega_2\left(f^{\prime\prime};\dfrac{1}{\sqrt{n+1}}\right)\right\}\\
	&+\!	\dfrac{1}{2}| L_1\!-\!a_1(n)|\cdot \| f^{\prime}\|_{\infty}
	+\dfrac{1}{2}|L_1\!+\!1|\left\{\dfrac{1}{n+1}\| f^{\prime}\|_{\infty}\!+\!\dfrac{1}{\sqrt{n+1}}\omega_1\left(f^{\prime};\dfrac{1}{\sqrt{n+1}}\right)+\dfrac{9}{8}\omega_2\left(f^{\prime};\dfrac{1}{\sqrt{n+1}}\right)\right\},
	\end{align*}
	where $X:=x(1-x)$, i.e., $X^{\prime}=1-2x$.
\end{theorem}
\begin{proof}
	For $\Delta_n^K$ the following inequality holds
	\begin{align}\label{ec2} \Delta_n^K&\leq \left|n\left[K_n(f;x)-f(x)\right]-\dfrac{X^{\prime}}{2}f^{\prime}(x)-\dfrac{X}{2}f^{\prime\prime}(x) \right|\nonumber\\
	&+
	\left|n\left[K_n^{M,1}(f;x)-K_n(f;x)\right]-\dfrac{X^{\prime}}{2}(1+L_1)f^{\prime}(x)\right|. \end{align}
	If $a_1(n)=-1$, i.e., $L_1=-1$, the second summand cancels. So we have the "old" Voronovskaya-Kantorovich theorem with second modulus (see \cite{AcGo}):
	\begin{align}\label{ec3}
	&	\left|n\left[K_n(f;x)-f(x)\right]-\dfrac{X^{\prime}}{2}f^{\prime}(x)-\dfrac{X}{2}f^{\prime\prime}(x) \right|\leq \dfrac{2}{3(n+1)}\left(\dfrac{3}{4}\|f^{\prime}\|_{\infty}+\|f^{\prime\prime}\|_{\infty}\right)\nonumber\\
	&+\dfrac{9}{32}\left\{\dfrac{2}{\sqrt{n+1}}\omega_1\left(f^{\prime\prime};\dfrac{1}{\sqrt{n+1}}\right)+\omega_2\left(f^{\prime\prime};\dfrac{1}{\sqrt{n+1}}\right)\right\}.
	\end{align}
	The second summand of (\ref{ec2}) can be written as
	\begin{align*}
	& n\left[K_n^{M,1}(f;x)-K_n(f;x)\right]-\dfrac{X^{\prime}}{2}(1+L_1)f^{\prime}(x)\\
	&=n\left\{(n+1)\sum_{k=0}^np_{n,k}^{M,1}(x)\int_{\frac{k}{n+1}}^{\frac{k+1}{n+1}}f(t)dt-(n+1)\sum_{k=0}^np_{n,k}(x)\int_{\frac{k}{n+1}}^{\frac{k+1}{n+1}}f(t)dt-\dfrac{X^{\prime}}{2}(L_1+1)f^{\prime}(x)\right\}\\
	&=n(n+1)\sum_{k=0}^{n-1}(a_1(n)+1)\left(x-\frac{1}{2}\right)p_{n-1,k}(x)\left[\int_{\frac{k+1}{n+1}}^{\frac{k+2}{n+1}}f(t)dt-\int_{\frac{k}{n+1}}^{\frac{k+1}{n+1}}f(t)dt\right]-\dfrac{X^{\prime}}{2}(L_1+1)f^{\prime}(x)\\
	&=-\dfrac{X^{\prime}}{2}(a_1(n)+1)(K_nf)^{\prime}(x)-\dfrac{X^{\prime}}{2}(L_1+1)f^{\prime}(x).
	\end{align*}
	So,
	\begin{align*}
	&	\left|n\left[K_n^{M,1}(f;x)-K_n(f;x)\right]-\dfrac{X^{\prime}}{2}(1+L_1)f^{\prime}(x)\right|\leq\dfrac{1}{2}|X^{\prime}|\left|(K_nf)^{\prime}(x)(a_1(n)+1)-(L_1+1)f^{\prime}(x)\right|\\
	&=\dfrac{1}{2}| X^{\prime}|\left|(K_nf)^{\prime}(x)(a_1(n)-L_1)+(L_1+1)\left[(K_nf)^{\prime}(x)-f^{\prime}(x)\right]\right|\\
	&\leq\dfrac{1}{2}\left\{|L_1-a_1(n)||(K_nf)^{\prime}(x)|+|L_1+1||(K_nf)^{\prime}(x)-f^{\prime}(x)|\right\}.
	\end{align*}
	But,
	\begin{align*}
	|(K_nf)^{\prime}(x)|&=n(n+1)\left|\sum_{k=0}^{n-1}p_{n-1,k}(x)\left[\int_{\frac{k+1}{n+1}}^{\frac{k+2}{n+1}}f(t)dt-\int_{\frac{k}{n+1}}^{\frac{k+1}{n+1}}f(t)dt\right]\right|\\
	&=n(n+1)\sum_{k=0}^{n-1}p_{n-1,k}(x)\int_{\frac{k+1}{n+1}}^{\frac{k+2}{n+1}}\left|f(t)-f\left(t-\dfrac{1}{n+1}\right)\right|dt\\
	&\leq\dfrac{n}{n+1}\sum_{k=0}^{n-1}\| f^{\prime}\|_{\infty}p_{n-1,k}(x)\leq \| f^{\prime}\|_{\infty}.
	\end{align*}
	Moreover, from \cite[Theorem 7]{GHR} it follows
	$$|(K_nf)^{\prime}(x)-f^{\prime}(x)|\leq \dfrac{1}{n+1}\| f^{\prime}\|_{\infty}+\dfrac{1}{\sqrt{n+1}}\omega_1\left(f^{\prime};\dfrac{1}{\sqrt{n+1}}\right)+\dfrac{9}{8}\omega_2\left(f^{\prime};\dfrac{1}{\sqrt{n+1}}\right). $$
	From the above relation we obtain
	\begin{align}\label{ec4}
	&\left|n\left[K_n^{M,1}(f;x)-K_n(f;x)\right]-\dfrac{X^{\prime}}{2}(1+L_1)f^{\prime}(x)\right|\leq \dfrac{1}{2}| L_1-a_1(n)|\cdot \| f^{\prime}\|_{\infty}\nonumber\\
	&+\dfrac{1}{2}|L_1+1|\left\{\dfrac{1}{n+1}\| f^{\prime}\|_{\infty}+\dfrac{1}{\sqrt{n+1}}\omega_1\left(f^{\prime};\dfrac{1}{\sqrt{n+1}}\right)+\dfrac{9}{8}\omega_2\left(f^{\prime};\dfrac{1}{\sqrt{n+1}}\right)\right\}.
	\end{align}
	Using the relations (\ref{ec2})-(\ref{ec4}) we get the claim.
\end{proof}

\begin{cor} We have
	$$\Delta_n^K\leq\left\{\begin{array}{l} 
	o(1)+\dfrac{1}{2}|L_1-a_1(n)|\cdot\|f^{\prime}\|_{\infty},\textrm{ for } f\in C^2[0,1],\\
	\vspace{-0.4cm}\\
	{\cal O}\left(\dfrac{1}{\sqrt{n}}\right)+\dfrac{1}{2}\left|L_1-a_1(n) \right|\cdot \| f^{\prime}\|_{\infty},\textrm{ for } f\in C^3[0,1],\\
	\vspace{-0.4cm}\\
	{\cal O}\left(\dfrac{1}{n}\right) +\dfrac{1}{2}\left|L_1-a_1(n) \right|\cdot \| f^{\prime}\|_{\infty},\textrm{ for } f\in C^4[0,1].
	\end{array}\right.  $$
\end{cor}

\section{The modified Durrmeyer operators $D_n^{M,1}$ }

The classical Durrmeyer operators  were introduced by Durrmeyer \cite{1a} and, independently, by Lupa\c s \cite{2a}. These operators are defined as
\begin{eqnarray*}
D_n^{M}(f;x) =(n+1) \sum\limits_{k=0}^{n} p_{n,k}(x)  \int\limits_{0}^{1} p_{n,k}(t) \ f(t) \, dt , \ \ x \in [0,1].
\end{eqnarray*}

In this section we study a Durrmeyer variant of  the modified
Bernstein operators introduced in a recent note of Acu, Gupta and Tachev \cite{AGT}:

\begin{eqnarray}\label{X}
D_n^{M,1}(f;x) =(n+1) \sum\limits_{k=0}^{n} p_{n,k}^{M,1}(x)  \int\limits_{0}^{1} p_{n,k}(t) \ f(t) \, dt , \ \ x \in [0,1].
\end{eqnarray}

\begin{theorem}\label{TD3.1} For $n\geq 1$ and $f\in C^2[0,1]$, one has
	\begin{align}\label{ec3}
	&	\left\|n\left(D_nf-f\right)-\left(Xf^{\prime}\right)^{\prime} \right\|_{\infty}\leq \dfrac{1}{n+2}\left(2\|f^{\prime}\|_{\infty}+3\|f^{\prime\prime}\|_{\infty}\right)\nonumber\\
	&+\dfrac{5}{\sqrt{n+4}}\omega_1\left(f^{\prime\prime};\dfrac{1}{\sqrt{n+4}}\right)+\dfrac{9}{8}\omega_2\left(f^{\prime\prime};\dfrac{1}{\sqrt{n+4}}\right),
	\end{align}
	where $X=x(1-x)$ and $X^{\prime}=1-2x$, $x\in[0,1]$.
\end{theorem}
\begin{proof}
	From \cite[Theorem 3]{GonskaRasa} we get
	\begin{align*}
	&\left| D_n(f;x)-f(x)-D_n(t-x;x)f^{\prime}(x)-\dfrac{1}{2}D_n\left((e_1-x)^2;x\right)f^{\prime\prime}(x)\right|\\
	&\leq D_n((e_1-x)^2;x)\left\{\dfrac{|D_n((e_1-x)^3;x)|}{D_n((e_1-x)^2;x)}\dfrac{5}{6h}\omega_1(f^{\prime\prime};h)+\left(\dfrac{3}{4}+\dfrac{D_n((e_1-x)^4;x)}{D_n((e_1-x)^2;x)}\cdot\dfrac{1}{16h^2}\right)\omega_2(f^{\prime\prime};h)\right\}.
	\end{align*}
	Using the central moments up to order 4 for Durrmeyer operators, namely
	\begin{align*}
	&D_n\left(t-x;x\right)=\dfrac{1-2x}{n+2},\\
	& D_n\left((t-x)^2;x\right)=\dfrac{2\left[x(1-x)(n-3)+1\right]}{(n+2)(n+3)},\\
	&D_n\left((t-x)^3;x\right)=\dfrac{6(1-2x)}{(n+2)(n+3)(n+4)}\left[2x(1-x)n+2x^2-2x+1\right],\\
	&D_n\left((t-x)^4;x\right)=\dfrac{12\left[x^2(1-x)^2n^2+3x(1-x)(7x^2-7x+2)n-10x(1-x)(x^2-x+1)+2\right]}{(n+2)(n+3)(n+4)(n+5)},
	\end{align*}
	we obtain
	\begin{align*}
	\dfrac{|D_n\left((t-x)^3;x\right)|}{D_n\left((t-x)^2;x\right)}\leq \dfrac{6}{n+4};\,\,\,\, \dfrac{|D_n\left((t-x)^4;x\right)|}{D_n\left((t-x)^2;x\right)}\leq \dfrac{6}{n+4}.
	\end{align*}
	Therefore, the following inequality holds
	\begin{align*}
	&	\left|D_n(f;x)-f(x)-\dfrac{1-2x}{n+2}f^{\prime}(x)-\dfrac{x(1-x)(n-3)+1}{(n+2)(n+3)}f^{\prime\prime}(x)\right|\\
	&\leq \dfrac{1}{(n+2)}\left\{\dfrac{5}{h(n+4)}\omega_1(f^{\prime\prime};h)+\left(\dfrac{3}{4}+\dfrac{3}{8h^2(n+4)}\right)\omega_2(f^{\prime\prime};h)\right\}
	\end{align*}
	and for $h=\dfrac{1}{\sqrt{n+4}}$ we obtain, after multiplying both sides by $n$,
	\begin{align*}
	&	\left|n\left[D_n(f;x)-f(x)\right]-\dfrac{n(1-2x)}{n+2}f^{\prime}(x)-\dfrac{n\left[x(1-x)(n-3)+1\right]}{(n+2)(n+3)}f^{\prime\prime}(x)\right|\\
	&\leq \dfrac{5}{\sqrt{n+4}}\omega_1\left(f^{\prime\prime};\dfrac{1}{\sqrt{n+4}}\right)+\dfrac{9}{8}\omega_2\left(f^{\prime\prime};\dfrac{1}{\sqrt{n+4}}\right).
	\end{align*}
	We can write
	\begin{align*}
	&\left|n\left[D_n(f;x)-f(x)\right]-X^{\prime}f^{\prime}(x)-Xf^{\prime\prime}(x)\right|\\
	&\leq\left|n\left[D_n(f;x)-f(x)\right]-\dfrac{n}{n+2}X^{\prime}f^{\prime}(x)-\dfrac{n(n-3)X}{(n+2)(n+3)} f^{\prime\prime}(x)-\dfrac{n}{(n+2)(n+3)}f^{\prime\prime}(x)\right|	 \\
	&+\left|X^{\prime}\left(\dfrac{n}{n+2}-1\right)f^{\prime}(x)+X\left[\dfrac{n(n-3)}{(n+2)(n+3)}-1\right]f^{\prime\prime}(x)+\dfrac{n}{(n+2)(n+3)}f^{\prime\prime}(x)\right|\\
	&\leq\dfrac{5}{\sqrt{n+4}}\omega_1\left(f^{\prime\prime};\dfrac{1}{\sqrt{n+4}}\right)+\dfrac{9}{8}\omega_2\left(f^{\prime\prime};\dfrac{1}{\sqrt{n+4}}\right)+\dfrac{1}{n+2}\left(2\|f^{\prime}\|_{\infty}+3\|f^{\prime\prime}\|_{\infty}\right).
	\end{align*}
\end{proof}

\begin{theorem}
	Let $f\in C[0,1]$, $x\in [0,1]$, $n\geq 1$. Then
	\begin{align*} \left| D_n^{M,1}(f;x)-f(x)\right|&\leq \left|D_n(f;x)-f(x)\right|
	\\&+ \left|(1+a_1(n))\left(\dfrac{1}{2}-x\right)\right| \left\{ 3\omega_2\left(f;\sqrt{\sigma_n(x)}\right)+\dfrac{5}{(n+2)\sqrt{\sigma_n(x)}}\omega_1(f;\sqrt{\sigma_n(x)})\right\}, \end{align*}
	where $\sigma_n(x)=\dfrac{2x(1-x)(n-1)(n-2)+3n+1}{2(n+2)^2(n+3)}$.
\end{theorem}
\begin{proof}
	We can write
	\begin{equation}\label{D1} \left| D_n^{M,1}(f;x)-f(x)\right|\leq \left|D_n(f;x)-f(x)\right|+\left| D_n^{M,1}(f;x)-D_n(f;x)\right|.\end{equation}
	Next, we will give an estimate of the quantity $\left| D_n^{M,1}(f;x)-D_n(f;x)\right|$. We have
	\begin{align*}
&	D_n^{M,1}(f;x)-D_n(f;x)=(n+1)\displaystyle\sum_{k=0}^n\{a(x,n)p_{n-1,k}(x)+a(1-x,n)p_{n-1,k-1}(x)\}\int_0^1p_{n,k}(t)f(t)dt\\
	&-(n+1)\displaystyle\sum_{k=0}^n\{(1-x)p_{n-1,k}(x)+xp_{n-1,k-1}(x)\}\int_0^1p_{n,k}(t)f(t)dt\\
	&=(n+1)(a_1(n)+1)\left(x-\dfrac{1}{2}\right)\left\{\displaystyle\sum_{k=0}^{n-1}p_{n-1,k}(x)\int_0^1p_{n,k}(t)f(t)dt\right.
	-\left.\displaystyle\sum_{k=1}^{n}p_{n-1,k-1}(x)\int_0^1p_{n,k}(t)f(t)dt\right\}\\
	&=(a_1(n)+1)\left(x-\dfrac{1}{2}\right)\left[A_n(f;x)-B_n(f;x)\right],
	\end{align*}
	where
	\begin{align*}
	&A_n(f;x)=\displaystyle\sum_{k=0}^{n-1}p_{n-1,k}(x)F_k(f),\quad B_n(f;x)=\displaystyle\sum_{k=0}^{n-1}p_{n-1,k}(x)G_k(f),\\
&	F_k(f;x)=(n+1)\displaystyle\int_0^1p_{n,k}(t)f(t)dt,\quad	G_k(f;x)=(n+1)\displaystyle\int_0^1p_{n,k+1}(t)f(t)dt,\,\, k=0,\dots,n-1.\\
	\end{align*}
		For a positive linear functional $F$ denote 
	$$ b^F:=F(e_1)\textrm{ and } \mu_2^F:=\dfrac{1}{2}F\left(e_1-b^Fe_0\right)^2. $$
	Using \cite[Theorem 5]{AnaRasa} for $f\in C[0,1]$ and $0<h\leq \dfrac{1}{2}$, we get
	\begin{equation}\label{C} \left|A_n(f;x)-B_n(f;x)\right|\leq\dfrac{3}{2}\left(1+\dfrac{\sigma_n(x)}{h^2}\right)\omega_2(f,h)+\dfrac{5\delta}{h}\omega_1(f,h), \end{equation}
	where \begin{align*}\sigma_n(x):=\displaystyle\sum_{k=0}^{n-1}\left(\mu_2^{F_k}+\mu_2^{G_k}\right)p_{n-1,k}(x),\,\,\, \delta:=\sup_{k}\left|b^{F_k}-b^{G_k}\right|,\\
	\end{align*}
In the present case
$$ b^{F_k}=\dfrac{k+1}{n+2},\,\,b^{G_k}=\dfrac{k+2}{n+2},\,\,\mu_2^{F_k}=\dfrac{(k+1)(n-k+1)}{2(n+2)^2(n+3)},\,\, \mu_2^{G_k}=\dfrac{(k+2)(n-k)}{2(n+2)^2(n+3)}, $$
so we obtain $\sigma_n(x)=\dfrac{2x(1-x)(n-1)(n-2)+3n+1}{2(n+2)^2(n+3)}$ and $\delta=\dfrac{1}{n+2}$.

Choosing $h:=\sqrt{\sigma_n(x)}$ we get
\begin{equation} \label{E1}\left| D_n^{M,1}(f;x)\!-\!D_n(f;x)\right|\!\leq\! \left|(1\!+\!a_1(n))\left(\dfrac{1}{2}\!-\!x\right)\right| \left\{ 3\omega_2\left(f;\sqrt{\sigma_n(x)}\right)\!+\!\dfrac{5}{(n\!+\!2)\sqrt{\sigma_n(x)}}\omega_1(f;\sqrt{\sigma_n(x)})\right\}.\end{equation}
Using relations (\ref{D1}) and (\ref{E1}) the proof is complete.
	
	\end{proof}

\begin{remark} \begin{itemize}
		\item[i)] For $a_1(n)=-1$ all the estimates for the Durrmeyer operator $D_n$ hold.
		\item[ii)] If $f\in C^2[0,1]$, then for $a_1(n)$ bounded $\|D_n^{M,1}(f)-f\|_{\infty}={\cal O}\left(\dfrac{1}{n}\right)$. 
	\end{itemize}
\end{remark}

\begin{theorem}
	Suppose that $D_n^{M,1}$ is given as above, $f\in C^2[0,1]$, $L_1=\displaystyle\lim_{n\to\infty}a_1(n)$ exists. Then for $x\in[0,1]$ there holds
	\begin{align*}
	&\Delta_n^D:=\left|n\left[D_n^{M,1}(f;x)-f(x)\right]-x(1-x)f^{\prime\prime}(x)-\dfrac{1-2x}{2}(L_1+3)f^{\prime}(x)\right|\\
	&\leq \dfrac{1}{n+2}\left(2\|f^{\prime}\|_{\infty}+3\|f^{\prime\prime}\|_{\infty}\right)
	+\dfrac{5}{\sqrt{n+4}}\omega_1\left(f^{\prime\prime};\dfrac{1}{\sqrt{n+4}}\right)+\dfrac{9}{8}\omega_2\left(f^{\prime\prime};\dfrac{1}{\sqrt{n+4}}\right)\\
	&+\!\dfrac{1}{2}\left\{|L_1\!-\!a_1(n)|\cdot\|f^{\prime}\|_{\infty}\!+\!|1\!+\!L_1|\left[\dfrac{2}{n+2}|f^{\prime}(x)|\!+\!\sqrt{\dfrac{2}{n\!+\!2}}\omega_1\left(f^{\prime};\sqrt{\dfrac{2}{n+2}}\right)\!+\!\dfrac{9}{8}\omega_2\left(f^{\prime};\sqrt{\dfrac{2}{n\!+\!2}}\right)\right]\right\}.
	\end{align*}
\end{theorem}
\begin{proof}
	For $\Delta_n^D$ the following inequality holds
	\begin{align}\label{YY1}
	\Delta_n^D&\leq\left|n\left[D_n(f;x)-f(x)\right]-x(1-x)f^{\prime\prime}(x)-(1-2x)f^{\prime}(x)\right|\nonumber\\
	&+\left|n\left[D_n^{M,1}(f;x)-D_n(f;x)\right]-(1-2x)\dfrac{L_1+1}{2}f^{\prime}(x)\right|.
	\end{align}
	The second difference of (\ref{YY1}) can be estimated as follows
	\begin{align*}
	&\left|n\left[D_n^{M,1}(f;x)-D_n(f;x)\right]-\dfrac{1-2x}{2}(1+L_1)
	f^{\prime}(x)\right|\\
	&=\left| n(n+1)(a_1(n)+1)\left(x-\dfrac{1}{2}\right)\left\{\displaystyle\sum_{k=0}^{n-1}p_{n-1,k}(x)\int_0^1p_{n,k}(t)f(t)dt-\sum_{k=1}^{n}p_{n-1,k-1}(x)\int_0^1p_{n,k}(t)f(t)dt
	\right\}\right.\\
	&-\left.\dfrac{1-2x}{2}(1+L_1)
	f^{\prime}(x)\right|.\\
	\end{align*}
	But,
	\begin{align*}
	(D_nf)^{\prime}(x)&=n\displaystyle\sum_{k=0}^{n-1}p_{n-1,k}(x)\int_0^1p_{n+1,k+1}(t)f^{\prime}(t)dt\\
	&=n(n+1)\left\{\displaystyle\sum_{k=1}^{n}p_{n-1,k-1}(x)\int_0^1p_{n,k}(t)f(t)dt
	-\sum_{k=0}^{n-1}p_{n-1,k}(x)\int_0^1p_{n,k}(t)f(t)dt.
	\right\}.	\end{align*}
	From the above relation, we get
	\begin{align}\label{YY3}&	\left|n\left[D_n^{M,1}(f;x)-D_n(f;x)\right]-\dfrac{1-2x}{2}(1+L_1)
	f^{\prime}(x)\right|\nonumber\\
	&\leq \left| (a_1(n)+1)\left(\dfrac{1}{2}-x\right)(D_nf)^{\prime}(x)-\dfrac{1-2x}{2}(1+L_1)f^{\prime}(x)  \right|\nonumber\\
	&\leq \left|\dfrac{1}{2}-x\right|\left\{|L_1-a_1(n)||(D_nf)^{\prime}(x)|+|1+L_1||f^{\prime}(x)-(D_nf)^{\prime}(x)|\right\}.
	\end{align}
	From \cite[Theorem 2.45]{Daniela} we have 
	\begin{align}\label{Y4}
	\left|(D_nf)^{\prime}(x)-f^{\prime}(x) \right|&\leq \left|\left(D_ne_1\right)^{\prime}(x)-1\right|\left|f^{\prime}(x)\right|+\dfrac{1}{h}\gamma(x)\omega_1\left(f^{\prime};h\right)\nonumber\\
	&+\left[\left(D_ne_1\right)^{\prime}(x)+\dfrac{1}{2h^2}\beta(x)\right]\omega_2(f^{\prime};h),
	\end{align}
	where
	\begin{align*}
	&\gamma(x):=\left|\left(D_n\left(\dfrac{1}{2}e_2-xe_1\right)\right)^{\prime}(x)\right|=\dfrac{2n|1-2x|}{(n+2)(n+3)}\leq \dfrac{2}{n+2};\\
	&\beta(x):=\left(D\left(\dfrac{1}{3}e_3-xe_2+x^2e_1\right)\right)^{\prime}(x)=\dfrac{2n\left[x(1-x)(n-11)+3\right]}{(n+2)(n+3)(n+4)}\leq\dfrac{1}{2(n+2)}.
	\end{align*}
	Choosing $h:=\sqrt{\dfrac{2}{n+2}}$, we get
	\begin{equation}\label{YY4}  \left|(D_nf)^{\prime}(x)-f^{\prime}(x) \right|\leq \dfrac{2}{n+2}|f^{\prime}(x)|+\sqrt{\dfrac{2}{n+2}}\omega_1\left(f^{\prime};\sqrt{\dfrac{2}{n+2}}\right)+\dfrac{9}{8}\omega_2\left(f^{\prime};\sqrt{\dfrac{2}{n+2}}
	\right).\end{equation}
	Also, we have
	\begin{equation}\label{YY5}
	\left|(D_nf)^{\prime}(x)\right|\leq n\| f^{\prime}\|_{\infty}\sum_{k=0}^{n-1}p_{n-1,k}(x)\int_0^1p_{n+1,k+1}(t)dt=\dfrac{n}{n+2}\| f^{\prime}\|_{\infty}\leq \| f^{\prime}\|_{\infty}.
	\end{equation}
	Using the relations (\ref{YY3}), (\ref{YY4}) and (\ref{YY5}), we get
		\begin{align}\label{YY6}&	\left|n\left[D_n^{M,1}(f;x)-D_n(f;x)\right]-\dfrac{1-2x}{2}(1+L_1)
	f^{\prime}(x)\right|
	\leq \dfrac{1}{2}\left\{|L_1-a_1(n)|\cdot\|f^{\prime}\|_{\infty}\right.\nonumber\\
	&+\left.|1+L_1|\left[\dfrac{2}{n+2}|f^{\prime}(x)|+\sqrt{\dfrac{2}{n+2}}\omega_1\left(f^{\prime};\sqrt{\dfrac{2}{n+2}}\right)+\dfrac{9}{8}\omega_2\left(f^{\prime};\sqrt{\dfrac{2}{n+2}}\right)\right]\right\}.
	\end{align}
	From the relations (\ref{YY1}), (\ref{YY6}) and Theorem \ref{TD3.1} the proof is complete.
\end{proof}

\begin{cor} We have
	$$\Delta_n^D\leq\left\{\begin{array}{l} {\cal O}\left(\dfrac{1}{\sqrt{n}}\right)+\dfrac{1}{2}\left|L_1-a_1(n) \right|\cdot \| f^{\prime}\|_{\infty},\textrm{ for } f\in C^3[0,1],\\
	\vspace{-0.4cm}\\
	{\cal O}\left(\dfrac{1}{n}\right) +\dfrac{1}{2}\left|L_1-a_1(n) \right|\cdot \| f^{\prime}\|_{\infty},\textrm{ for } f\in C^4[0,1].
	\end{array}\right.  $$
\end{cor}

\section{The modified genuine Bernstein-Durrmeyer operators $U_n^{M,1}$}
The genuine Bernstein-Durrmeyer operators were introduced by Chen \cite{D2} and Goodman and Sharma \cite{dif_GBD} as follows:
\begin{align*} &{ U}_{n}(f;x)=(1-x)^n f(0)+ x^n f(1)\\
&+(n-1)
\displaystyle\sum_{k=1}^{n-1}\left(\int_{0}^1 f(t)p_{n-2,k-1}(t)dt\right)p_{n,k}(x),\,\,  f\in C[0,1]. \end{align*}

Using the fundamental polynomials $p_{n,k}^{M,1}$ modified  genuine Bernstein-Durrmeyer operators can be introduced as follows:
\begin{align}\label{B}
{ U}_n^{M,1}(f;x)&=a(x,n)(1-x)^{n-1}f(0)+a(1-x,n)x^{n-1}f(1)\nonumber\\
&+(n-1)\displaystyle\sum_{k=1}^{n-1}p_{n,k}^{M,1}(x)\int_0^1p_{n-2,k-1}(t)f(t)dt.
\end{align}
This modification was also investigated in a recent note of Acu and Agrawal \cite{AcAg}. All the results given there will be improved in this section.
Throughout this section we assume $U_n^{M,1}(e_0)=1$, namely the sequences $a_0(n)$ and $a_1(n)$ verify the condition (\ref{A}).

\begin{theorem}
	Let $f\in C[0,1]$, $x\in [0,1]$, $n\geq 1$. Then
	\begin{align*} \left| U_n^{M,1}(f;x)-f(x)\right|&\leq \left|U_n(f;x)-f(x)\right|
	\\&+ \left|(1+a_1(n))\left(\dfrac{1}{2}-x\right)\right| \left\{ 3\omega_2\left(f;\sqrt{\sigma_n(x)}\right)+\dfrac{5}{n\sqrt{\sigma_n(x)}}\omega_1(f;\sqrt{\sigma_n(x)})\right\}, \end{align*}
	where $\sigma_n(x)=\dfrac{\left[2nx(1-x)+(1-2x)^2\right](n-1)}{n^2(n+1)}\leq\dfrac{1}{4n}$.
\end{theorem}
\begin{proof}
	We have
	\begin{equation}\label{D} \left| U_n^{M,1}(f;x)-f(x)\right|\leq \left|U_n(f;x)-f(x)\right|+\left| U_n^{M,1}(f;x)-U_n(f;x)\right|.\end{equation}
	In the following we will give an estimate of the quantity $\left| U_n^{M,1}(f;x)-U_n(f;x)\right|$. So,
	\begin{align*}
	U_n^{M,1}(f;x)-U_n(f;x)&=(n-1)\displaystyle\sum_{k=1}^{n-1}\left\{a(x,n)p_{n-1,k}(x)+a(1-x,n)p_{n-1,k-1}(x)\right\}\int_0^1p_{n-2,k-1}(t)f(t)dt\\
	&+a(x,n)(1-x)^{n-1}f(0)+a(1-x,n)x^{n-1}f(1)-(1-x)^nf(0)-x^nf(1)	\\
	&-(n-1)\displaystyle\sum_{k=1}^{n-1}\left\{(1-x)p_{n-1,k}(x)+xp_{n-1,k-1}(x)\right\}\int_0^1p_{n-2,k-1}(t)f(t)dt\\
	&=\left[(a_1(n)+1)x+a_0(n)-1\right]\left\{(n-1)\displaystyle\sum_{k=1}^{n-1}p_{n-1,k}(x)\int_0^1p_{n-2,k-1}(t)f(t)dt\right.\\
	&-\left.(n-1)\sum_{k=1}^{n-1}p_{n-1,k-1}(x)\int_0^1p_{n-2,k-1}(t)f(t)dt
	+(1-x)^{n-1}f(0)-x^{n-1}f(1)\right\}\\
	&=\left[(a_1(n)+1)x+a_0(n)-1\right]\left(A_n(f;x)-B_n(f;x)\right),\end{align*}
	where
	\begin{align*}
&	A_n(f;x):=(n-1)\displaystyle\sum_{k=1}^{n-1}p_{n-1,k}(x)\int_0^1p_{n-2,k-1}(t)f(t)dt
		+(1-x)^{n-1}f(0);\\
	&	B_n(f;x):=(n-1)\displaystyle\sum_{k=0}^{n-2}p_{n-1,k}(x)\int_0^1p_{n-2,k}(t)f(t)dt
	+x^{n-1}f(1).\\
	\end{align*}
	Note that the operators $A_n$ and $B_n$ can be written as follows
	$$A_n(f;x)=\displaystyle\sum_{k=0}^{n-1}F_k(f)p_{n-1,k}(x),\quad B_n(f;x)=\displaystyle\sum_{k=0}^{n-1}G_k(f)p_{n-1,k}(x),  $$
	where
	\begin{align*}
	& F_0(f;x)=f(0),\,\, F_k(f;x)=(n-1)\int_0^1p_{n-2,k-1}(t)f(t),\,\, k=1,\dots, n-1,\\
	&G_k(f;x)=(n-1)\int_0^1p_{n-2,k}(t)f(t)dt,\,k=0,\dots, n-2,\quad G_{n-1}(f;x)=f(1).
	\end{align*}
Let	 $F$ be a positive linear functional and
	\begin{align*} & b^F:=F(e_1),\,\,\, \mu_2^F:=\dfrac{1}{2}F\left(e_1-b^Fe_0\right)^2, \\
	&\sigma_n(x):=\displaystyle\sum_{k=0}^{n-1}\left(\mu_2^{F_k}+\mu_2^{G_k}\right)p_{n-1,k}(x),\,\,\, \delta:=\sup_{k}\left|b^{F_k}-b^{G_k}\right|.
	\end{align*}
	Using \cite[Theorem 5]{AnaRasa} for $f\in C[0,1]$ and $0<h\leq \dfrac{1}{2}$, we get
	\begin{equation}\label{C} \left|A_n(f;x)-B_n(f;x)\right|\leq\dfrac{3}{2}\left(1+\dfrac{\sigma_n(x)}{h^2}\right)\omega_2(f,h)+\dfrac{5\delta}{h}\omega_1(f,h).\end{equation}
	Since
$$ b^{F_k}=\dfrac{k}{n},\,\,b^{G_k}=\dfrac{k+1}{n},\,\,\mu_2^{F_k}=\dfrac{1}{2}\dfrac{k(n-k)}{n^2(n+1)},\,\, \mu_2^{G_k}=\dfrac{1}{2}\dfrac{(k+1)(n-k-1)}{n^2(n+1)}, $$
we get $\sigma_n(x)=\dfrac{\left[2nx(1-x)+(1-2x)^2\right](n-1)}{2n^2(n+1)}\leq\dfrac{1}{4n}$ and $\delta=\dfrac{1}{n}$.

Choosing $h:=\sqrt{\sigma_n(x)}$ we obtain
\begin{equation} \label{E}\left| U_n^{M,1}(f;x)-U_n(f;x)\right|\leq \left|(1+a_1(n))\left(\dfrac{1}{2}-x\right)\right| \left\{ 3\omega_2\left(f;\sqrt{\sigma_n(x)}\right)+\dfrac{5}{n\sqrt{\sigma_n(x)}}\omega_1(f;\sqrt{\sigma_n(x)})\right\}.\end{equation}
Using relations (\ref{D}) and (\ref{E}) we obtain the inequality claimed.
\end{proof}

\begin{remark} \begin{itemize}
		\item[i)] For $a_1(n)=-1$  estimates for genuine Bernstein-Durrmeyer operator $U_n$ are obtained; details are omitted here.
			\item[ii)] If $f\in C^2[0,1]$, then for $a_1(n)$ bounded $\|U_n^{M,1}(f)-f\|_{\infty}={\cal O}\left(\dfrac{1}{n}\right)$. 
	\end{itemize}
\end{remark}

\begin{theorem}
	Suppose that $U_n^{M,1}$ is given as above, $f\in C^2[0,1]$, $L_1=\displaystyle\lim_{n\to\infty}a_1(n)$ exists. Then for $x\in[0,1]$ there holds
	\begin{align*}
	\Delta_n^U&:=\left|n\left[U_n^{M,1}(f;x)-f(x)\right]-x(1-x)f^{\prime\prime}(x)-\dfrac{1-2x}{2}(1+L_1)f^{\prime}(x)\right|\\
	&\leq \displaystyle\dfrac{5\sqrt{6}}{12}\omega_1\left(f^{\prime\prime};\sqrt{\dfrac{3}{n+2}}\right)+\dfrac{13}{32}\omega_2\left(f^{\prime\prime};\sqrt{\dfrac{3}{n+2}}\right)+\dfrac{9}{8}\omega_2\left(f;\sqrt{\dfrac{2}{n+1}} \right)\\
	&+\dfrac{1}{2}\left\{|L_1-a_1(n)|\|f^{\prime}\|_{\infty}+|1+L_1|\left[\dfrac{1}{\sqrt{n+1}}\omega_1\left(f^{\prime};\dfrac{1}{\sqrt{n+1}}\right)+\dfrac{5}{4}\omega_2\left(f^{\prime};\dfrac{1}{\sqrt{n+1}}\right)\right]\right\}.
	\end{align*}
\end{theorem}
\begin{proof}
	For $\Delta_n^U$ the following inequality holds:
	\begin{align}\label{Y1}
	\Delta_n^U&\leq\left|n\left[U_n(f;x)-f(x)\right]-x(1-x)f^{\prime\prime}(x)\right|\nonumber\\
	&+\left|n\left[U_n^{M,1}(f;x)-U_n(f;x)\right]-\dfrac{1-2x}{2}(1+L_1)f^{\prime}(x)\right|.
	\end{align}
	From \cite[Theorem 5]{GonskaRasa} a quantitative Voronovskaya-type theorem for genuine Bernstein-Durrmeyer operators can be given as follows:
	\begin{align*}
	&\left|(n+1)\left[U_n(f;x)\!-\!f(x)\right]-x(1-x)f^{\prime\prime}(x)\right|\leq\displaystyle\dfrac{5\sqrt{6}}{12}\omega_1\left(f^{\prime\prime};\sqrt{\dfrac{3}{n\!+\!2}}\right)\!+\!\dfrac{13}{32}\omega_2\left(f^{\prime\prime};\sqrt{\dfrac{3}{n+2}}\right),n\geq 1.
	\end{align*}
	Using the pointwise estimate of genuine Bernstein-Durrmeyer operator (see \cite[Corollary 3.25]{Daniela})
	$$ |U_n(f;x)-f(x)|\leq \dfrac{9}{8}\omega_2\left(f;\sqrt{\dfrac{2}{n+1}} \right),$$
	we get
	\begin{align}
	\left|n\left[U_n(f;x)-f(x)\right]-x(1-x)f^{\prime\prime}(x)\right|&\leq \displaystyle\dfrac{5\sqrt{6}}{12}\omega_1\left(f^{\prime\prime};\sqrt{\dfrac{3}{n+2}}\right)+\dfrac{13}{32}\omega_2\left(f^{\prime\prime};\sqrt{\dfrac{3}{n+2}}\right)\nonumber\\
	&+\dfrac{9}{8}\omega_2\left(f;\sqrt{\dfrac{2}{n+1}} \right).\label{Y2}
	\end{align}
	The second difference of (\ref{Y1}) can be estimated as follows
	\begin{align*}
&\left|n\left[U_n^{M,1}(f;x)-U_n(f;x)\right]-\dfrac{1-2x}{2}(1+L_1)
f^{\prime}(x)\right|\\
&=\left| n(a_1(n)+1)\left(x-\dfrac{1}{2}\right)\left\{(1-x)^{n-1}f(0)-x^{n-1}f(1)+(n-1)\displaystyle\sum_{k=1}^{n-1}p_{n-1,k}(x)\int_0^1p_{n-2,k-1}(t)f(t)dt\right.\right.\\
&-\left.\left.(n-1)\sum_{k=1}^{n-1}p_{n-1,k-1}(x)\int_0^1p_{n-2,k-1}(t)f(t)dt
\right\}-\dfrac{1-2x}{2}(1+L_1)
f^{\prime}(x)\right|.\\
	\end{align*}
	But,
	\begin{align*}
	(U_nf)^{\prime}(x)&=n\displaystyle\sum_{k=0}^{n-1}p_{n-1,k}(x)\int_0^1p_{n-1,k}(t)f^{\prime}(t)dt\\
	&=n\left\{-(1-x)^{n-1}f(0)+x^{n-1}f(1)-(n-1)\displaystyle\sum_{k=1}^{n-1}p_{n-1,k}(x)\int_0^1p_{n-2,k-1}(t)f(t)dt\right.\\
	&+\left.(n-1)\sum_{k=1}^{n-1}p_{n-1,k-1}(x)\int_0^1p_{n-2,k-1}(t)f(t)dt
	\right\}.	\end{align*}
	From the above relation, we get
\begin{align}\label{Y3}&	\left|n\left[U_n^{M,1}(f;x)-U_n(f;x)\right]-\dfrac{1-2x}{2}(1+L_1)
	f^{\prime}(x)\right|\nonumber\\
	&\leq \left| (a_1(n)+1)\left(\dfrac{1}{2}-x\right)(U_nf)^{\prime}(x)-\dfrac{1-2x}{2}(1+L_1)f^{\prime}(x)  \right|\nonumber\\
	&\leq \left|\dfrac{1}{2}-x\right|\left\{|L_1-a_1(n)||(U_nf)^{\prime}(x)|+|1+L_1||f^{\prime}(x)-(U_nf)^{\prime}(x)|\right\}.
\end{align}
From \cite[Theorem 3.14]{Daniela} we have 
\begin{equation}\label{Y4}
\left|(U_nf)^{\prime}(x)-f^{\prime}(x) \right|\leq\dfrac{1}{\sqrt{n+1}}\omega_1\left(f^{\prime};\dfrac{1}{\sqrt{n+1}}\right)+\dfrac{5}{4}\omega_2\left(f^{\prime};\dfrac{1}{\sqrt{n+1}}\right).
\end{equation}
Also, there holds
\begin{equation}\label{Y5}
\left|(U_nf)^{\prime}(x)\right|\leq n\| f^{\prime}\|_{\infty}\sum_{k=0}^{n-1}p_{n-1,k}(x)\int_0^1p_{n-1,k}(t)dt=\| f^{\prime}\|_{\infty}.
\end{equation}
From the relations (\ref{Y1})-(\ref{Y5}) the proof is complete.
\end{proof}

\begin{cor} We have
	$$\Delta_n^U\leq\left\{\begin{array}{l} {\cal O}\left(\dfrac{1}{\sqrt{n}}\right)+\dfrac{1}{2}\left|L_1-a_1(n) \right|\cdot \| f^{\prime}\|_{\infty},\textrm{ for } f\in C^3[0,1],\\
	\vspace{-0.4cm}\\
	{\cal O}\left(\dfrac{1}{n}\right) +\dfrac{1}{2}\left|L_1-a_1(n) \right|\cdot \| f^{\prime}\|_{\infty},\textrm{ for } f\in C^4[0,1].
	\end{array}\right.  $$
\end{cor}

$$    $$

\noindent{\bf Acknowledgements.} The first author acknowledges the support of  Lucian Blaga University of Sibiu under  research grant LBUS-IRG-2018-04. The second one is grateful for the departmental facilities provided during his senior professorship at the University of Duisburg-Essen.

$   $

\end{document}